\documentclass[a4paper,11pt]{article}
\usepackage{authblk}
\usepackage[english]{babel} 
\usepackage[centertags,intlimits]{amsmath}
\usepackage{amsfonts}
\usepackage{amsthm}
\usepackage{amssymb}
\usepackage{comment}
\usepackage{pgf, tikz, pgfplots, float, graphicx, tikz-cd}
\usetikzlibrary{decorations.pathreplacing,calligraphy, arrows}

\newtheorem{theorem}{Theorem}[section]
{Proposition}
\newtheorem{definition}{Definition}[section]
\newtheorem{lemma}[theorem]%
{Lemma}
\newtheorem{corollary}[theorem]%
{Corollary}

\newtheorem{example}{Example}[section]

\begin{document}

\title{Remoteness, order, size and connectivity constraints in digraphs}

\author[]{Sufiyan Mallu
\thanks{Financial support by the South African National Research Foundation is greatly acknowledged.}}
\affil[]{University of Johannesburg\\
South Africa}

\maketitle

\begin{abstract}



Let $D$ be a strongly connected digraph. The average distance of a vertex $v \in V(D)$ is the arithmetic mean distance from $v$ to all other vertices in $D$. The remoteness $\rho(D)$ is the maximum average distance among the vertices of $D$. 

In this paper we establish a sharp upper bound on the remoteness of a strong digraph in terms of its order, size, and vertex-connectivity, and characterise the extremal digraphs attaining this bound. Furthermore, we show that known upper bounds for an undirected graphs - under size and connectivity constraints \cite{DanMafMal2025} or minimum degree conditions \cite{Dan2015} - extend to the broader class of Eulerian digraphs, which includes all graphs.
\end{abstract}

MSC-class: 05C12, 05C20

\section{Introduction}

While distances in graphs have been the subject of extensive study, the investigation of distances in digraphs remains relatively underdeveloped. In particular, the concept of remoteness have not been explored as thoroughly in the directed setting as they have in the undirected case. Ai et al. \cite{AiGerGutMaf2021} were the first to extend these notions to digraphs, establishing foundational results. This paper aims to contribute to the growing literature on distances in digraphs by providing new results on the maximum remoteness of strong digraphs under certain constraints. To the best of our knowledge, this is the second paper to focus specifically on the study of remoteness in the context of directed graphs.\\

Let $D$ be a strong connected, finite digraph of order $n$. In a digraph $D$ of order at least two, the \emph{average distance} $\overline{\sigma}(v, D)$ of a vertex $v$ is defined as the arithmetic mean of the directed distances from $v$ to all other vertices of $D$, that is,
\[
\overline{\sigma}(v, D) = \frac{1}{n-1} \sum_{w \in V(D)} d_D(v, w),
\]
where $d_D(v, w)$ denotes the length of the shortest directed path from $v$ to $w$.  
The \emph{remoteness} $\rho(D)$ of a strongly connected digraph $D$ is then defined as the maximum of the average distances of its vertices, namely
\[
\rho(D) = \max_{v \in V(D)} \overline{\sigma}(v, D),
\]
where $V(D)$ is the vertex set of $D$.  Furthermore, the \emph{average distance} $\mu(D)$ of a strong connected digraph $D$ is defined as the mean of the distances over all ordered pairs of distinct vertices, that is,
\[
\mu(D)
= \frac{1}{n(n-1)} \sum_{\substack{u,v \in V(D)}} d_D(u,v),
\]
where $V(D)$ denotes the vertex set of $D$. So the average distance $\mu(D)$ of a strong digraph $D$ of order $n$ is equal to the mean of the average distances of its vertices, that is,
\[
\mu(D) = \frac{1}{n} \sum_{v \in V(D)} \overline{\sigma}(v,D).
\]

The above definitions can be viewed as the directed analogue of the corresponding distance-based notion in graphs.

The term \emph{remoteness} first appeared in a paper titled \emph{automated comparison of graph invariants} by Aouchiche, Caporossi, and Hansen \cite{AouCapHan2007}, and has since been widely studied in the literature. Nevertheless, the concept and related ideas had been explored earlier under different terminology. For instance, Zelinka \cite{Zel1968} investigated the \emph{vertex deviation}, defined as $\frac{\sigma(v, G)}{n}$, where $\sigma(v, G)$ is the sum of distances from vertex $v$ to all other vertices in the graph, and $n$ is the total number of vertices. Other researchers have referred to $\sigma(v, G)$ using terms such as \emph{transmission} (e.g.,\cite{Ple1984}), \emph{total distance}, or simply \emph{distance} of a vertex $v$.

Bounds on remoteness in terms of order only were given by Zelinka \cite{Zel1968} and later, independently, by Aouchiche and Hansen \cite{AouHan2011}. 

\begin{theorem}\label{Zel1968AouHan2011}
{\rm (Zelinka \cite{Zel1968}, Aouchiche, Hansen \cite{AouHan2011})} \\
Let $G$ be a connected graph of order $n\geq 2$. Then
\[ \rho(G) \leq \frac{n}{2}, \]
with equality if and only if $G$ is a path.
\end{theorem}

The bound in Theorem \ref{Zel1968AouHan2011} was extended to digraphs by Ai et al. \cite{AiGerGutMaf2021}.

\begin{theorem}\label{AiGerGutMaf2021}
{\rm (Ai, Gerke, Gutin, Mafunda \cite{AiGerGutMaf2021})}\\
Let $D$ be a strong digraph of order $n\geq 3$. Then 
\[\rho(D)\leq \frac{n}{2},\]
with equality if and only if $D$ is strong and contains a Hamiltonian dipath 
$v_1v_2\dots v_n$ such that no directed edge of the form
$v_iv_j$ with $2\le i+1<j\le n$ is in $D$.
\end{theorem}

Dankelmann \cite{Dan2015} improved Theorem \ref{Zel1968AouHan2011} by establishing sharper upper bound on the remoteness of graphs of given order and minimum degree. In particular, the bound is improved by a factor of about $\frac{3}{\delta+1}$ for graphs with minimum degree $\delta \geq 2$.

\begin{theorem}\label{Dan2015}
{\rm (Dankelmann \cite{Dan2015})}\\
Let $ G $ be a connected graph of order $ n $ and minimum degree $ \delta \; \geq \; 2 $. Then
\begin{align*}
\rho(G) \; \leq& \; \dfrac{3n}{2(\delta+1)} + \dfrac{7}{2}.
\end{align*}
\end{theorem}

In \cite{EntJacSny1976}, Entringer, Jackson, and Snyder strengthen the results of Theorem \ref{Zel1968AouHan2011} by incorporating the size of a graph.

\begin{theorem}\label{EntJacSny1976}
{\rm (Entringer, Jackson, Snyder~\cite{EntJacSny1976})}\\
Let $G$ be a connected graph of order $n$ and size at least $m$. Then
\[
\rho(G) \leq \frac{n+2}{2} - \frac{m}{n-1}.
\]
\end{theorem}

Recently, Dankelmann et al. \cite{DanMafMal2025} demonstrated that the bound in Theorem~\ref{EntJacSny1976} can be significantly improved for $\kappa$-connected graphs, where $\kappa$ is arbitrary, and for $\lambda$-edge-connected graphs with $\lambda \in \{2,3\}$.

\begin{theorem}\label{DanMafMal2025} 
{\rm (Dankelmann, Mafunda, Mallu~\cite{DanMafMal2025})}\\
Let $G$ be a $\kappa$-connected graph of order $n$ and size at least $m$ with $m \leq \binom{n-1}{2}$. Then 
\begin{equation*} 
\rho(G) \leq \frac{n}{2\kappa} +2 - \frac{1}{\kappa} - \frac{\kappa-1}{n-1} - \frac{m}{\kappa(n-1)}.
\end{equation*}
The bound is sharp. 
\end{theorem}

\begin{theorem}\label{DankMafMal2025}
{\rm (Dankelmann, Mafunda, Mallu~\cite{DanMafMal2025})}\\
(a) Let $G$ be a $2$-edge-connected graph of order $n$ and size $m$. 
Then
\[ \rho(G) \leq \left\{ \begin{array}{cc}
   \frac{n}{3} & \textrm{if $m <\lceil \frac{5}{3}n \rceil -2$,} \\
   \frac{n}{3} - \frac{2m}{3(n-1)} +\frac{5}{3} 
         & \textrm{if $m  \geq \lceil \frac{5}{3}n \rceil -2$},
      \end{array} \right. \]
and this bound is sharp apart from an additive constant. \\
(b) Let $G$ be a $3$-edge-connected graph of order $n$ and size $m$. 
Then
\[ \rho(G) \leq \left\{ \begin{array}{cc}
   \frac{n}{4} & \textrm{if $m < \lceil \frac{9}{4}n\rceil -2$,} \\
   \frac{n}{4} - \frac{m}{2(n-1)} + \frac{3}{2} 
        & \textrm{if $m \geq \lceil \frac{9}{4}n\rceil -2$,}
      \end{array} \right. \]
and this bound is sharp apart from an additive constant. 
\end{theorem}

The goal of this paper is to show that the bound in Theorem \ref{DanMafMal2025} can be strengthened significantly for $\kappa$-connected strong digraphs.  Furthermore we show that the bounds in Theorem \ref{Dan2015}, Theorem \ref{DanMafMal2025} and Theorem \ref{DankMafMal2025} can be extended to a larger class of digraphs containing all undirected graphs, the Eulerian digraphs. A strong digraph is said to be Eulerian if for every vertex its in-degree and its out-degree are equal.

The literature contains several results on the remoteness of graphs, ranging from bounds 
on the remoteness of various classes of graphs to the relationships between remoteness and other 
graph parameters. There are results on remoteness in  outerplanar graphs \cite{DanMafMal20251}, in graphs that forbid certain cycles \cite{DanJonMaf2021}, and in trees \cite{AllPreHefMaf2026, Zel1968}. 
Relations between remoteness and other graph parameters have also been explored for example, maximum degree \cite{DanMafMal2022}, 
and clique number \cite{HuaDas2014}. Differences between remoteness and other graph parameters 
have also been studied; see, for example, \cite{AouHan2011, Dan2016, DanMafMal20252}.  A survey on proximity and remoteness in graphs is provided in \cite{Aou2024}. 

 While the results in this paper concern upper bounds on remoteness under given order, size, and connectivity constraints, related extremal questions have also been investigated from a complementary perspective. In particular, upper bounds on the maximum possible size of graphs and digraphs subject to restrictions on distance parameters have been obtained for graphs, Eulerian digraphs, and general digraphs; see, for example, \cite{Cam2023,Dan2021}.

This paper is organised as follows. In Section 2, we introduce the terminology and notation used throughout. In Section 3, we establish an upper bound on the remoteness of a strong digraph with given order, size, and vertex-connectivity, and characterise the extremal digraphs that attain this bound. In Section 4, we show that the upper bounds obtained on remoteness of an undirected graphs of given order and either size and connectivity constraints  \cite{DanMafMal2025} or minimum degree \cite{Dan2015}  can be generalised to a larger class of digraphs that includes all graphs, in particular the Eulerian digraphs.

\section{Terminology and notation}
We use the following notation. For a digraph $D$ we denote by $V(D)$ and $E(D)$ the {\em vertex set} and {\em edge set} (often also called arc set), respectively.  The {\em order} and {\em size} of $D$ are denoted by $n(D)$ and $m(D)$, respectively. For vertices $u$ and $v$ in $D$, the arc from $u$ to $v$ is denoted by $\overrightarrow{uv}$.
 By an $(n,m)$-digraph we mean a digraph of order $n$ and size at least $m$. $D$ is strong if for every pair $u,v$ of vertices, $D$ contains both a $u-v$ path and a $v-u$ path. The {\em vertex-connectivity} $\kappa(D)$ of a digraph $D$ is the minimum number of vertices whose removal results in a digraph that is not strong. The {\em edge-connectivity} $\lambda(D)$ of a digraph $D$ is the minimum number of arcs that must be removed to make the digraph not strong.

If $v$ is a vertex in a digraph $D$, the {\em out-neighbourhood} of $v$ is defined as $N_{D}^+(v) = \{u \in V(D) \mid \overrightarrow{vu} \in E(D)\}$, and the {\em in-neighbourhood} of $v$ is given by $N_{D}^-(v) = \{u \in V(D) \mid \overrightarrow{uv} \in E(D)\}$. The {\em out-degree}, denoted by $d^+_D(v)$, and the {\em in-degree}, denoted by $d^-_D(v)$, correspond to the cardinalities $|N_{D}^+(v)|$ and $|N_{D}^-(v)|$, respectively. The {\em degree of a vertex} $v$ is simply ${\rm deg}(v)=d^+_D(v)+d^-_D(v)$. The {\em minimum degree} $\delta(D)$ of $D$ is defined as the minimum among the in-degrees and out-degrees of all vertices of $D$.

For a strong digraph \( D \) and $v \in V(D) $, the {\em distance} \( d_D(v, w) \) is defined as the minimum number of arcs on a path from \(v\) to \(w\). The {\em eccentricity} ${\rm ecc}_D(v)$ of a vertex $v$ is the distance from $v$ to a vertex farthest from $v$. The largest of all eccentricities of vertices of $D$ is called the {\em diameter} and is denoted by 
${\rm diam}(D)$. Furthermore, let $ N_{i}(v)=\{x \in V(D)|d_{D}(v,x)=i \} $ and $ |N_{i}(v)|=n_i(v) $ for $ i \in \mathbb{Z}^{+} $.  Clearly, $n_i(v) >0$ if and only if $0\leq i \leq {\rm ecc}_D(v)$. 
The {\em distance degree} of $v$ is the sequence $(n_{0}(v),n_{1}(v),n_{2}(v),\ldots,n_{d}(v))$ where $d = {\rm ecc}_D(v)$ and is denoted by $X_{D}(v)$. Let $ N_{\leq i}(v)=\{x \in V(D)|d_{D}(v,x)\leq i \} $ and  $ N_{\geq i}(v)=\{x \in V(D)|d_{D}(v,x)\geq i \} $.

We denote the complete graph of order $n$ by $K_n$. If $ D_1,D_2,\ldots, D_k $ are disjoint digraphs, then the {\em sequential sum} $ D_1 + D_2 + \ldots + D_k $ is the digraph obtained from their union by adding an arc from every vertex in $ D_i $ to every vertex in $ D_{i+1} $ and from every vertex in $ D_{i+1} $ to every vertex in $ D_{i} $ for $ i=1, 2, \ldots, k-1 $. The sequential sum of undirected graphs is defined analogously. If $G$ is an undirected graph then by $ \overleftrightarrow{G} $ we mean the digraph that is obtained from $G$ by replacing each undirected edge by two arcs in the opposite direction. Hence, by $\overleftrightarrow{K_n}$ we mean a {\em  complete digraph} of order $n$. If $t, k\in \mathbb{N}$, then $[K_{a_1} + K_{a_2}+\ldots+K_{a_t}]^k$ 
stands for $k$ repetitions of the pattern 
$K_{a_1} + K_{a_2}+\ldots+K_{a_t}$. If $D$ is a digraph, then the {\em complement} of $D$, denoted by $\overline{D}$, is the digraph on the same vertex set as $D$, in which two vertices are adjacent if
they are not adjacent in $D$. 

If $D$ and $D'$ are distinct digraphs with the same vertex set, but $A(D) \subset A(D')$, i.e., $D'$ is obtained from $D$ by adding arcs, we will denote this relationship by $D \lneqq D'$. If $(k,b)$ and $(k',b')$ are distinct pairs of integers, then we write $(k,b) \prec (k',b')$ if $(k,b)$ comes before $(k',b')$ in the lexicographic ordering of pairs of integers, i.e., if
either $k<k'$ or $k=k'$ and $b< b'$. With the necessary adjustments, all the above definitions also hold for graphs.

\section{Maximum remoteness of a strong digraph with given order and size}\label{strong digraph}
In this section, we determine the maximum remoteness of $\kappa$-connected strong digraphs for a given order, size, and vertex connectivity. While the proof technique is partially inspired by the work of Dankelmann et al.~\cite{DanMafMal2025}, the digraph setting presents significantly greater challenges. The arguments developed here demand a more detailed analysis, as well as the construction of more sophisticated extremal example. This section begins with the introduction of the necessary definitions.

\begin{definition} \label{def:kappa-connected-path-complete}
{\rm (Dankelmann, Mafunda, Mallu~\cite{DanMafMal2025})}\\
A graph $G$ is said to be a $\kappa$-connected path-complete graph if there exist
$\ell, a, b \in \mathbb{N}$, $a \geq \kappa$, with  
\[ G = K_1 + [K_{\kappa}]^{\ell} + K_a + K_b.\]
\end{definition}

For given $n,m \in \mathbb{N}$ for which there exists a $\kappa$-connected path-complete graph of order $n$ and size at least $m$, Dankelmann et al. \cite{DanMafMal2025} defined $PK_{n,m,\kappa}$ to be such a graph of minimum size. Note that for graphs of diameter greater than $2$, Definition \ref{def:kappa-connected-path-complete} generalises the path-complete graphs defined by Solt\'{e}s in \cite{Sol1991}, which are the $1$-connected path-complete graphs. We now define a directed analogue of the $\kappa$-connected path-complete graph.

\begin{definition}
Let \( D_1, D_2, \ldots, D_k \) be digraphs. We define the digraph $$D_1 \overleftarrow{+} D_2 \overleftarrow{+} D_3 \overleftarrow{+} \cdots \overleftarrow{+} D_k$$ as the digraph obtained from the disjoint union of \( D_1, D_2, \ldots, D_k \) by adding directed arcs in both directions between each vertex in $D_i$ and each vertex in $D_{i+1}$ for $i=1,2,\ldots,k-1$, and directed arcs from each vertex in $D_i$ to each vertex in $D_j$ for each \( i = 3, 4, \ldots, k \) and for all $j \in \{1,2,\ldots,i-2\}$.
\end{definition}

\begin{definition}\label{def:kappa-connected-path-complete digraph}
Let $\ell,a,b\in \mathbb{N}$ with $a \geq \kappa$. A \textbf{$\kappa$-connected path-complete digraph} is a digraph of the form \[ K_1 \; \overleftarrow{+} \; [ \overleftrightarrow{K_\kappa}]^{\ell} \; \overleftarrow{+} \; \overleftrightarrow{K_a} \; \overleftarrow{+} \; \overleftrightarrow{K_b} .\]
\end{definition}
\noindent See Figure \ref{fig:The vertex connected path-complete strong digraph} for an example of a $\kappa$-connected path-complete digraph with $\ell=2$.
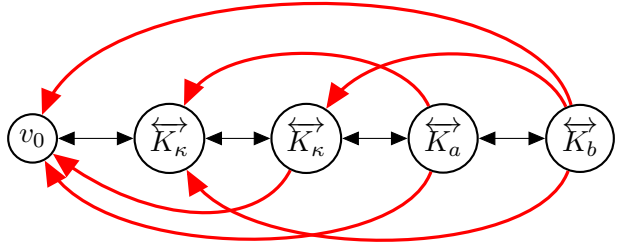
\begin{figure}[H]
	\begin{center}
		\begin{tikzpicture}
		[scale=0.3,inner sep=0.8mm, 
		vertex/.style={circle,thick,draw}, 
		thickedge/.style={line width=10pt}] 
		\begin{scope}[>=triangle 45]
		
		\node[vertex] (a1) at (-8,0) [fill=white] {$v_{0}$};
		\node[vertex] (a2) at (-2,0) [fill=white] {$\overleftrightarrow{K_{\kappa}}$};
		\node[vertex] (a3) at (4,0) [fill=white] {$\overleftrightarrow{K_{\kappa}}$}; 
		\node[vertex] (a4) at (10,0) [fill=white] {$\overleftrightarrow{K_{a}}$};
	\node[vertex] (a5) at (16,0) [fill=white]  {$\overleftrightarrow{K_{b}}$};

       \draw[<->, black](a1)--(a2);
       \draw[<->, black](a2)--(a3);
       \draw[<->, black](a3)--(a4);
       \draw[<->, black](a4)--(a5);
       
       \draw[->, very thick, red] (a5)..controls (14,4.5) and (7,4.5)..(a3);
       \draw[->, very thick, red] (a5)..controls (14,-5.5) and (1,-5.5)..(a2);
       \draw[->, very thick, red] (a5)..controls (14,7.5) and (-5,7.5)..(a1);
       
         \draw[->, very thick, red] (a4)..controls (8,4.5) and (0,4.5)..(a2);
       \draw[->, very thick, red] (a4)..controls (8,-5.5) and (-5,-5.5)..(a1);
       
       \draw[->, very thick, red] (a3)..controls (2,-4) and (-2.75,-4)..(a1);
       
       
       \end{scope}
		\end{tikzpicture}
		\caption{The $\kappa$-connected path-complete digraph with $\ell=2$.}\label{fig:The vertex connected path-complete strong digraph}
	\end{center} 
	\end{figure}

Clearly, a $\kappa$-connected path-complete digraph is strong $\kappa$-connected and has diameter  at least 3. Furthermore, we refer to a $1$-connected path-complete digraph simply as a \textbf{\textit{path-complete digraph}}.

\begin{lemma}\label{la:size-vs-remoteness-in-kappa-connected-pc-digraph}
	(a) Let $H$ be a $\kappa$-connected path-complete digraph, where $\kappa \in \mathbb{N}$. Then 
$\rho(H+\overrightarrow{uv}) < \rho(H)$ for any arc $ \overrightarrow{uv} \in E(\overline{H})$. \\
(b) Let $H_1, H_2$ be two distinct $\kappa$-connected path-complete strong digraphs of order $n$. Then either $m(H_1) < m(H_2)$ and $\rho(H_1) > \rho(H_2)$, or $m(H_1) > m(H_2)$ 
and $\rho(H_1) < \rho(H_2)$. \\
 (c) Given $n$, $\kappa$. Then there exists a $\kappa$-connected path-complete strong digraph of order $n$ and size $m$ if an only if $m \equiv (n^2-2n-1) \pmod{\kappa} $ and $ \frac{n}{2}(3\kappa+n)-n-\kappa^2-b(\kappa-b) \leq m \leq n^2-2n-1$, where $b \in \{1,2,\ldots, \kappa\}$ with
$b \equiv n-1 \pmod{\kappa}$.
\end{lemma}
\begin{proof}
(a) Let $a,b,\kappa \in \mathbb{N}$ with $H = K_1 \; \overleftarrow{+} \; [ \overleftrightarrow{K_\kappa}]^{\ell} \; \overleftarrow{+} \; \overleftrightarrow{K_a} \; \overleftarrow{+} \; \overleftrightarrow{K_b}$ and let $v_0$ be the vertex in the leftmost $K_1$. Then, clearly $v_0$ uniquely attains the remoteness of $H$. Now, adding any new arc to $H$ reduces the distance from $v_0$ to some vertex, and so it reduces $\sigma(v_0, H)$ and thus the remoteness, thereby establishing part (a) of the lemma.\\

(b) Let $H_1$ and $H_2$ be two distinct $\kappa$-connected path-complete strong digraphs of order $n$. By (a) it suffices to show 
\begin{equation} \label{eq:either-H<H' or H'<H kappa digraph}
H_2 \lneqq H_1 \quad \textrm{or} \quad H_1 \lneqq H_2.
\end{equation}

To prove \eqref{eq:either-H<H' or H'<H kappa digraph}, let $\ell, \ell', a, a', b, b' \in \mathbb{N}$ with $a, a' \geq \kappa$ such that 
\[ H_1 =K_1 \; \overleftarrow{+} \; [ \overleftrightarrow{K_\kappa}]^{\ell} \; \overleftarrow{+} \; \overleftrightarrow{K_a} \; \overleftarrow{+} \; \overleftrightarrow{K_b}, \quad 
  H_2 = K_1 \; \overleftarrow{+} \; [ \overleftrightarrow{K_\kappa}]^{\ell'} \; \overleftarrow{+} \; \overleftrightarrow{K_{a'}} \; \overleftarrow{+} \; \overleftrightarrow{K'_{b'}}, \]  
with $\ell$ and $\ell'$ repetitions of $\overleftrightarrow{k_{\kappa}}$ for $H$ and $H'$, respectively. \\

Since $H_1 \neq H_2$, it follows that $(\ell,b)\neq (\ell',b')$. We have either  
$(\ell,b) \prec (\ell',b')$ or $(\ell',b') \prec (\ell,b)$. 
Without loss of generality we may assume the former. 
First assume that $\ell=\ell'$ and $b <b'$. Then $H_1$ is obtained from $H_2$ by 
adding arcs from the rightmost complete digraph $\overleftrightarrow{K_{\kappa}}$ to $b'-b$ vertices of the complete digraph $\overleftrightarrow{K_{b'}}$, and so we have that $H_2 \lneqq H_1$. Now assume that \(\ell < \ell'\).
We obtain \(H_1\) from \(H_2\) by adding arcs as follows: first, add arcs from all vertices of the first rightmost complete digraph $\overleftrightarrow{K_{\kappa}}$ to all vertices of the complete digraph \(\overleftrightarrow{K_{b'}}\) in \(H_2\), resulting in the $\kappa$-connected strong digraph \(K_1 \; \overleftarrow{+} \; [ \overleftrightarrow{K_\kappa}]^{\ell'} \; \overleftarrow{+} \; \overleftrightarrow{K_{a'+b'}}\). From this point onwards, we successively add arcs  from all vertices of the second rightmost complete digraph $\overleftrightarrow{K_{\kappa}}$ to all vertices of the last rightmost complete digraph \(\overleftrightarrow{K}_{q_{i}}\), where $q_i=a'+b'+i\kappa$ for $i \in \{1,2,\ldots,\ell'-\ell-1\}$ and $i$ denote the $i^{th}$ iteration. This yields the $\kappa$-connected path-complete digraphs \(K_1 \; \overleftarrow{+} \; [ \overleftrightarrow{K_\kappa}]^{\ell'-1} \; \overleftarrow{+} \; \overleftrightarrow{K_{q_1}}\), \(K_1 \; \overleftarrow{+} \; [ \overleftrightarrow{K_\kappa}]^{\ell'-2} \; \overleftarrow{+} \; \overleftrightarrow{K_{q_2}}\), and so on until after repeating this process $\ell'-\ell-1$ times, we obtain the $\kappa$-connected path-complete digraph
\(K_1 \; \overleftarrow{+} \; [ \overleftrightarrow{K_\kappa}]^{\ell + 1 } \; \overleftarrow{+} \; \overleftrightarrow{K_{q_{\ell'-\ell-1}}}\). Note that $a'+b'+(\ell'-\ell-1)\kappa = a+b-\kappa$. Finally, by adding arcs from the second rightmost complete digraph $\overleftrightarrow{K_{\kappa}}$ to a subset of \(a - \kappa\) vertices of the complete digraph \(\overleftrightarrow{K_{a + b - \kappa}}\) in \(K_1 \; \overleftarrow{+} \; [ \overleftrightarrow{K_\kappa}]^{\ell + 1 }  \; \) \( \overleftarrow{+} \; \overleftrightarrow{K_{a+b-\kappa}}\), we obtain the digraph \(K_1 \; \overleftarrow{+} \; [ \overleftrightarrow{K_\kappa}]^{\ell} \; \overleftarrow{+} \; \overleftrightarrow{K_{a+b}}\), which is \(H_1\). Hence, \(H_2 \lneqq H_1\), and \eqref{eq:either-H<H' or H'<H kappa digraph} follows, confirming that part (b) holds.\\

(c) Fix \( n \) and $\kappa$. If for $\ell, a,b \in \mathbb{N}$ the strong digraph  
\( K_1 \; \overleftarrow{+} \; [ \overleftrightarrow{K_\kappa}]^{\ell} \; \overleftarrow{+} \; \overleftrightarrow{K_a} \; \overleftarrow{+} \; \) \(\overleftrightarrow{K_b} \) has order $n$ and is $\kappa$-connected, then $n=1 + \ell\kappa + a+b$, and 
$a \geq \kappa$. This implies that $\ell = \frac{n-1-a-b}{\kappa} \leq \frac{n-2-\kappa}{\kappa}$, and $b = n-1-\ell\kappa - a \leq n-1 - (\ell+1)\kappa$. With respect to the order $\prec$, the smallest and largest pairs $(\ell,b)$ satisfying 
these conditions are $(1,1)$ and $(\ell_0, b_0)$, respectively, where 
$\ell_0=\lfloor \frac{n-2-\kappa}{\kappa} \rfloor$ and $b_0=n-1- (\ell_0+1)\kappa$.  It thus follows as in the proof of \eqref{eq:either-H<H' or H'<H kappa digraph} that the 
$\kappa$-connected path-complete strong digraph $K_1 \; \overleftarrow{+} \; \overleftrightarrow{K_\kappa} \; \overleftarrow{+} \overleftrightarrow{K_{n-\kappa-2}} \; \overleftarrow{+} \; K_1$, arising from the pair $(1,1)$, has maximum size among all $\kappa$-connected path-complete strong digraphs of order $n$. Simple calculations show that its size is $n^2-2n-1$. The $\kappa$-connected path-complete strong digraph $K_1 \; \overleftarrow{+} \; [ \overleftrightarrow{K_\kappa}]^{\ell_0} \; \overleftarrow{+} \; \overleftrightarrow{K_{a_0}} \; \overleftarrow{+} \; \overleftrightarrow{K_{b_0}}$, arising from the pair
$(\ell_0,b_0)$ has minimum size among $\kappa$-connected path-complete strong digraphs of order $n$. Its size is $ \frac{n}{2}(3\kappa+n)-n-\kappa^2-b(\kappa-b)$.\\

The proof of part (b) shows that if $m(H_1) < m(H_2)$, then $H_2$ is obtained from $H_1$ by adding arcs, and the number of arcs added is a multiple of $\kappa$. Hence the number of arcs of a $\kappa$-connected path-complete strong digraph of order $n$ is congruent $n^2-2n-1$ modulo $\kappa$. If $H_1$ is a $\kappa$-connected path-complete strong digraph of order $n$, $H_1=K_1 \; \overleftarrow{+} \; [ \overleftrightarrow{K_\kappa}]^{\ell} \; \overleftarrow{+} \; \overleftrightarrow{K_a} \; \overleftarrow{+} \; \overleftrightarrow{K_b}$, then unless $(\ell,b)=(0,1)$, there exists a $\kappa$-connected path-complete strong digraph of order $n$ with exactly $\kappa$ more arcs than $H_1$: the strong digraph $K_1 \; \overleftarrow{+} \; [ \overleftrightarrow{K_\kappa}]^{\ell} \; \overleftarrow{+} \; \overleftrightarrow{K_{a+1}} \; \overleftarrow{+} \; \overleftrightarrow{K_{b-1}}$ (if $b>1$) or the strong digraph $K_1 \; \overleftarrow{+} \; [ \overleftrightarrow{K_\kappa}]^{\ell-1} \; $ $\overleftarrow{+} \; \overleftrightarrow{K_{a+1}}$ (if $\ell>0$ and $b=1$). This completes the proof of part (c).
\end{proof}

Given $n,m \in \mathbb{N}$ for which there exists a $\kappa$-connected
path-complete $(n,m)$-strong digraph, we define $DPK_{n,m,\kappa}$ to be such a graph
of minimum size. It follows from Lemma \ref{la:size-vs-remoteness-in-kappa-connected-pc-digraph}(a) 
that there exists at most one $\kappa$-connected path-complete strong digraph of given order 
and size, so $DPK_{n,m,\kappa}$ is well-defined. 

	\begin{theorem}
(a) Let $D$ be a $\kappa$-connected strong digraph of order $n$ and size $m$ with $ m \leq n^2-2n-1$. Then 
	\begin{equation} \label{eq:bound-on-rho-for-kappa-digraph}
\rho(D) \leq \rho(DPK_{n,m,\kappa}).
\end{equation}
(b) Assume that $m \equiv (n^2-2n-1) \pmod{\kappa} $ and $ \frac{n}{2}(3\kappa+n)-n-\kappa^2-b(\kappa-b) \leq m \leq n^2-2n-1$, where $b$ is the integer in $\{1,2,\ldots, \kappa\}$ with
$b \equiv n-1 \pmod{\kappa}$. Then equality in (a) holds only if $D=DPK_{n,m,\kappa}$.
	\end{theorem}
	
	\begin{proof}
We first prove that there exists a $\kappa$-connected path-complete $(n,m)$- strong digraph $D'$ with 
\begin{equation} \label{eq:exists-pc-kappa-digraph}
\rho(D) \leq \rho(D'). 
\end{equation}
We may assume that $D$ has maximum remoteness among all $\kappa$-connected $(n,m)$- strong digraphs, and that among all such strong digraphs with maximum remoteness, $D$ is one with the maximum size. Furthermore, let $v \in V(D)$ with $\overline{\sigma}(v,D) = \rho(D)$, $d = {\rm ecc}_D(v)$, $N_{i} = \{z \in V(D) \mid d_{D}(v,z) = i \}$, and $|N_{i}| = n_i$ for $i \in \{0, 1, \ldots, d\}$. Clearly, $n_0 = 1$,  and $\sum_{i=0}^{d} n_i = n$. \\

\emph{{\bf Claim 1:}} $D=K_{n_0} \; \overleftarrow{+} \; \overleftrightarrow{K_{n_1}} \; \overleftarrow{+} \ldots \; \overleftarrow{+} \;  \overleftrightarrow{K_{n_{d-1}}} \; \overleftarrow{+} \; \overleftrightarrow{K_{n_{d}}}$. \\
 
Recall that \( D \) has the maximum size among all strong digraphs of size at least \( m \) for which \( \sigma(v, D) \) is maximised. Consequently, each \( N_i \) induces a complete digraph \( \overleftrightarrow{K_{n_i}} \) in the digraph \( D \), otherwise we could add arcs between vertices of $N_i$ without changing the remoteness of $D$. Additionally, \( D \) has arcs from every vertex in \( N_i \) to every vertex in \( N_{i+1} \), and arcs from every vertex in \( N_{i+1} \) to every vertex in \( N_i \) for \( i = 0, 1, 2, \ldots, d-1 \). Furthermore, every vertex in \( N_j \) is adjacent to every vertex in \( N_i \) for all \( j > i \) where \( i, j \in \{0, 1, \ldots, d\} \). Hence, Claim 1 follows. \\

Note that $n_i \geq \kappa$ holds for $i=1,2,\ldots,d-1$, but not for $0$.\\

\emph{{\bf Claim 2:}} For all $ i \in \{1,2, \ldots, d-3\} $, we have $ n_i = \kappa $. \\

Suppose to the contrary that there exists $ j \in \{1,\ldots,d-3\} $ with $ n_{j}>\kappa $. Let $j$ be the smallest such value. Then $n_{i} = \kappa$ for all $ i \in \{1,\ldots,j-1\} $. Now consider the strong digraph $D^*$ that is obtained from $D$ by moving a vertex from $N_j$ to $N_{j+1}$, i.e., $D^*=K_{n'_0} \; \overleftarrow{+} \; \overleftrightarrow{K_{n'_1}} \; \overleftarrow{+} \ldots \; \overleftarrow{+} \; \overleftrightarrow{K_{n'_d}}$, where
$n'_i=n_i$ for $i \in \{0,1,\ldots,d\} - \{j,j+1\}$,  
$n'_j=n_j-1$ and $n'_{j+1}= n_{j+1}+1$. Then $m(D^*) = m(D) + 2(n_{j+2} - n_{j-1}) \geq m(D)$ since $n_{j+2} \geq \kappa=n_{j-1}$ and $D^{*}$ is $(n,m)$-strong digraph with $\sigma(v,D^{*}) > \sigma (v,D)$, and thus $\rho(D^{*}) > \rho(D)$. This contradiction to the maximality of $\rho(D)$ proves Claim 2.  \\

\emph{{\bf Claim 3:}} $ n_{d-2} = \kappa $. \\

Since $D$ is $k$-connected, we have that $n_{d-2} \geq \kappa$. Now, suppose to the contrary that \( n_{d-2} \neq \kappa \). Then \( n_{d-2} > \kappa \). Consider the strong digraph \( D^* \), which is obtained from \( D \) by moving one each vertex from \( N_{d-2} \) and \( N_d \) to \( N_{d-1} \), i.e., $D^*=K_{n'_0} \; \overleftarrow{+} \; \overleftrightarrow{K_{n'_1}} \; \overleftarrow{+} \ldots \; \overleftarrow{+} \; \overleftrightarrow{K_{n'_d}}$, where
$n'_i=n_i$ for $i \in \{0,1,\ldots,d-3\} - \{d-2,d-1,d\}$,  
$n'_{d-2}=n_{d-2}-1$, $n'_{d}=n_{d}-1$ and $n'_{d-1}= n_{d-1}+2$. It is easy to verify that $m(D^*) = m(D) + n_d - n_{d-3} + n_{d-2} - 1$. Since \( n_{d-3} = \kappa \) and \( n_{d-2} \geq \kappa+1 \), we have $
m(D^*) \geq m(D) + n_d > m(D).
$ Moreover, \( \rho(D^*) = \rho(D) \). This contradicts our choice of $D$ as a digraph with the maximum size among those of maximum remoteness. Thus, Claim 3 follows.  \\

\noindent It follows from Claims 1 to 3 that $D$ is a $\kappa$-connected path-complete strong digraph. Letting $D'=D$ proves (\ref{eq:exists-pc-kappa-digraph}).  \\

By (\ref{eq:exists-pc-kappa-digraph}), there exists a $\kappa$-connected path-complete strong digraph $D'$ of order $n$ and size at least $m$ with $\rho(D) \leq \rho(D')$. By the definition of $DPK_{n,m,\kappa}$, we have $m(D') \geq m(DPK_{n,m,\kappa})$.  By Lemma \ref{la:size-vs-remoteness-in-kappa-connected-pc-digraph}(b), it follows that $\rho(D') \leq \rho(DPK_{n,m,\kappa}) $. Hence 
\[\rho(D) \leq \rho(D') \leq \rho(DPK_{n,m,\kappa}),\]
which proves (a). \\

(b) Now assume that equality holds in \eqref{eq:bound-on-rho-for-kappa-digraph}, i.e., that
$\rho(D) = \rho(DPK_{n,m,\kappa})$, and furthermore that $m \equiv (n^2-2n-1) \pmod{\kappa} $ and $ \frac{n}{2}(3\kappa+n)-n-\kappa^2-b(\kappa-b) \leq m \leq n^2-2n-1$, where $b$ is as defined above. 
It follows from Lemma \ref{la:size-vs-remoteness-in-kappa-connected-pc-digraph}(c) that the 
graph $DPK_{n,m,\kappa}$ has exactly $m$ arcs. \\

It follows from part (b) that $D$ has maximum remoteness among all $\kappa$-connected 
$(n,m)$-strong digraphs. We claim that $D$ has maximum size among all such strong digraphs 
maximising the remoteness. Suppose not. Then there exists a $\kappa$-connected 
$(n,m+1)$-strong digraph $D''$ with $\rho(D) = \rho(D'')$. Applying \eqref{eq:bound-on-rho-for-kappa-digraph} to $D''$ we get that
\[ \rho(D) = \rho(D'') \leq \rho(DPK_{n,m+1,\kappa}) < \rho(DPK_{n,m,\kappa}), \]
where the last inequality follows from 
Lemma \ref{la:size-vs-remoteness-in-kappa-connected-pc-digraph}(b) and the fact that 
$m(DPK_{n,m,\kappa})= m < m(DPK_{n,m+1,\kappa})$. Hence $D$ has maximum size among all $\kappa$-connected strong digraphs of order $n$ maximising the remoteness. \\

The proof of (a) shows that, if $D$ has maximum size among all $\kappa$-connected
path-complete $(n,m)$-strong digraphs, then $D$ is a path-complete strong digraph. Hence $D=DPK_{n,m',\kappa}$ for some $m'$ with $m' \geq m$. Since by Lemma \ref{la:size-vs-remoteness-in-kappa-connected-pc-digraph} we have 
$\rho(DPK_{n,m',\kappa}) < \rho(DPK_{n,m,\kappa})$ if $m'>m$, we have $m'=m$, 
and thus $D=DPK_{n,m,\kappa}$, as desired. 
\end{proof}

Evaluating the remoteness of $DPK_{n,m,\kappa}$, yields the following corollary. Assume that $m \equiv (n^2-2n-1) \pmod{\kappa} $ and $ \frac{n}{2}(3\kappa+n)-n-\kappa^2-b(\kappa-b) \leq m \leq n^2-2n-1$, where $b \in \{1,2,\ldots, \kappa\}$ with
$b \equiv n-1 \pmod{\kappa}$.
	
\begin{corollary}\label{coro:evaluation-of-kappa-connecetd-path-complete-strong-digraph}
Let $D$ be a $\kappa$-connected strong digraph of order $n$ and size $m$, with $ \frac{n}{2}(3\kappa+n)-n-\kappa^2-b(\kappa-b) \leq m \leq n^2-2n-1$, where $b \in \{1,2,\ldots, \kappa\}$ with $b \equiv n-1 \pmod{\kappa}$. Let $m^{*}$ be the smallest integer with $m^* \geq m$ and $m^* \equiv (n^2-2n-1) \pmod{\kappa} $. Then 
\[ \rho(D) \leq  \frac{n}{\kappa}+2-\frac{1}{\kappa}-\frac{\kappa-1}{n-1}-\frac{m^*}{\kappa(n-1)}, \]
and this bound is sharp.  
\end{corollary}
 
\begin{proof}
Let $m^*$ be as defined above. It follows from Lemma \ref{la:size-vs-remoteness-in-kappa-connected-pc-digraph} that
the digraph $DPK_{n,m,\kappa}$ has size $m^*$. Note that $a+b=n-\ell\kappa-1$. Let $v_0$ be the vertex in first $K_1$ of $DPK_{n,m,\kappa}$, so $v_0$ realizes the remoteness. 
Let $H:= DPK_{n,m,\kappa}-V(\overleftrightarrow{K_a} \cup \overleftrightarrow{K_b})$. 
Straightforward calculations show that $n(H)=\ell\kappa+1$, 
$\sigma(v_0,H)=\frac{\ell\kappa}{2}(\ell+1)$ and 
$m(H)=\frac{1}{2}(\ell\kappa^2)(\ell+3)-\kappa(\kappa-1)$. Hence
\[\sigma(v_0,DPK_{n,m,\kappa})=\sigma(v_0,H)+(\ell+1)(a+b)+b=\frac{\kappa\ell}{2}(\ell+1)+(\ell+1)(a+b)+b,\]
and, since $m(DPK_{n,m,\kappa})=m^*$,  
\begin{align*}
m^* =& \; m(H)+2{a+b \choose 2}+2\kappa a +(a+b)\ell\kappa-(\kappa-1)a+b \\
=& \; \frac{1}{2}(\ell\kappa^2)(\ell+3)-\kappa(\kappa-1)+(a+b)(a+b-1)+2\kappa a +(a+b)\ell\kappa\\
&\;-(\kappa-1)a+b.
\end{align*}
Define $\epsilon=\rho(DPK_{n,m,\kappa})-\left(\frac{n}{\kappa}+2-\frac{1}{\kappa}-\frac{\kappa-1}{n-1}-\frac{m^*}{\kappa(n-1)}\right)$. 
Substituting the above terms for $\rho(DPK_{n,m,\kappa})$ and $m^*$, it is straightforward 
to verify that $\epsilon=0$. 
This proves that 
$\rho(DPK_{n,m,\kappa}) 
     =\frac{n}{\kappa}+2-\frac{1}{\kappa}-\frac{\kappa-1}{n-1}-\frac{m^*}{\kappa(n-1)}$,
and the corollary follows.
\end{proof}

For $\kappa=1$, Corollary \ref{coro:evaluation-of-kappa-connecetd-path-complete-strong-digraph} yields
the following corollary. Note that for $\kappa=1$ we have $m^*=m$.

\begin{corollary}   \label{coro:evaluation-of-path-complete digraph}
Let $D$ be a strong digraph of order $n$ and size $m$, with $ \frac{(n-1)(n+2)}{2} \leq m \leq n^2-2n-1$. Then 
\[ \rho(D) \leq  n+1-\frac{m}{n-1}, \]
and this bound is sharp.  
\end{corollary}

Note that the bound in the above corollary also holds for \( m < \frac{(n-1)(n+2)}{2} \), i.e., for smaller values of \( m \), and it is sharp; however, the extremal digraph is not unique.

\section{Extending bounds on remoteness of graphs for given size and connectivity constraints to Eulerian digraphs }

We note that the bounds obtained in this section improve those derived in Section 3 under the additional assumption that the digraph is Eulerian. Furthermore, our results extend known bounds on remoteness for undirected graphs, such as Theorem \ref{Dan2015}, Theorem \ref{DanMafMal2025}, and Theorem \ref{DankMafMal2025}, to Eulerian digraphs. Indeed, any connected graph can be regarded as an Eulerian digraph by replacing each edge with two arcs in the opposite direction. Under this interpretation, our bounds coincide with the classical graph bounds. Thus, the results in this section provide a natural generalisation of graph bounds to the Eulerian directed setting.

In this section, we establish sharp upper bounds on the remoteness of Eulerian digraphs of order $n$ and size at least $m$. We also derive corresponding bounds under connectivity constraints, including vertex connectivity and edge-connectivity for $\lambda \in \{2,3\}$. These results improve distance bounds for Eulerian digraphs compared to those known for strong, but not necessarily Eulerian, digraphs (see \cite{Dan2019, Dan2021, Osa2026}).

Dankelmann \cite{Dan2021} established an upper bound on the size of an Eulerian digraph for given diameter and order \( n \). 

\begin{theorem}\label{th:m(D)}
{\rm (Dankelmann \cite{Dan2021})}\\
Let $D$ be an Eulerian digraph of diameter $d$, and $v$ be a vertex of eccentricity $d$. If $X(v)=(n_0,\ldots,n_d)$ then
\[m(D) \leq 2m(K_{n_0}+K_{n_1}+\ldots+K_{n_d}),\]
equality holds if and only if $D=\overleftrightarrow{K_{n_0}} + \overleftrightarrow{K_{n_1}} + \ldots + \overleftrightarrow{K_{n_{d-2}}} + \overleftrightarrow{K_{n_{d-1}}} + \overleftrightarrow{K_{n_d}}$.
\end{theorem}


%

We now apply Theorem \ref{th:m(D)} in order to obtain sharp upper bounds on the maximum remoteness for Eulerian digraphs of given order, size and vertex connectivity from sharp bounds for graphs. We recall that, for given $n,m \in \mathbb{N}$ for which there exists a $\kappa$-connected path-complete graph of order $n$ and size at least $m$, $PK_{n,m,\kappa}$ denotes a $\kappa$-connected path-complete graph of minimum size, as defined in Section 3.



\begin{theorem}\label{theo:remoteness-of-kappa-connected-graphs}
{\rm (Dankelmann, Mafunda, Mallu~\cite{DanMafMal2025})}\\
Let $G$ be a $\kappa$-connected graph of order $n$ and size at least $m$ with $m \leq \binom{n-1}{2}$. Then 
\begin{equation*} 
\rho(G) \leq \rho (PK_{n,m,\kappa})=\frac{n}{2\kappa} +2 - \frac{1}{\kappa} - \frac{\kappa-1}{n-1} - \frac{m}{\kappa(n-1)}.
\end{equation*}
Equality holds only if $G= PK_{n,m,\kappa}$ and $m$ is congruent $\binom{n-1}{2}$ mod $\kappa$ and 
$\frac{1}{2} [ n(3\kappa-1) -2\kappa^2-\kappa+1 - b(\kappa-b) ] 
\leq m
\leq \binom{n-1}{2}$,
where $b \in \{1,2,\ldots,\kappa\}$ 
with $b \equiv \binom{n-1}{2} \pmod{\kappa}$.  
\end{theorem}

In the following Theorem we present a bound on the remoteness of $\kappa$-connnected Eulerian digraphs. To make the statement more comparable with Theorem \ref{theo:remoteness-of-kappa-connected-graphs}, and since every graph of size $m$ can be viewed as an Eulerian digraph of size $2m$, we state our result
for Eulerian digraphs of size at least $2m_0$.

\begin{theorem}\label{th:rho(vertex connectivity - D)}
Let $ D $ be a $\kappa$-connected Eulerian digraph of order $ n $ and size at least $ 2m_{0} $, where $m_0 \in \mathbb{N}$ . Then 
\[\rho(D) \leq \; \rho(\overleftrightarrow{PK_{n,m_0, \kappa}}) \; =\frac{n}{2\kappa} +2 - \frac{1}{\kappa} - \frac{\kappa-1}{n-1} - \frac{m_0}{\kappa(n-1)}. \]
The bound is sharp if $m_0 \equiv \binom{n-1}{2} \pmod{\kappa}$ and
$\frac{1}{2} [ n(3\kappa-1) -2\kappa^2-\kappa+1 - b(\kappa-b) ] 
\leq m_0
\leq \binom{n-1}{2}$,
where $b$ is the integer in $\{1,2,\ldots,\kappa\}$ 
with $b \equiv \binom{n-1}{2} \pmod{\kappa}$.
\end{theorem}
\begin{proof}
Let $D$ be an arbitrary $\kappa$-connected Eulerian digraph of order $n$ and size at least $2m_{0}$. Assume $v$ is chosen such that $\rho(D)=\overline{\sigma}(v,D)$. Let $X_{D}(v)=(n_{0},n_{1},n_{2},\ldots,n_{d})$ be the distance degree of $v$, where $d \in \mathbb{N}$. Furthermore, note that for all $ i \in \{1,2, \ldots, d-1\} $, we have that $ n_i \geq \kappa $, since each $N_i$ is a cutset of $D$. \\ 

\noindent Let $D'$ be the digraph $\overleftrightarrow{K_{n_0}} + \overleftrightarrow{K_{n_1}} + \ldots + \overleftrightarrow{K_{n_{d-2}}} + \overleftrightarrow{K_{n_{d-1}}} + \overleftrightarrow{K_{n_d}}$. Clearly the digraph $D'$ is Eulerian with the same order as $D$ and is $\kappa$-connected since $X_{D}(v)=X_{D'}(v)$. Furthermore, $\rho(D')=\rho(D)$ and by Theorem \ref{th:m(D)} we have that $m(D') \geq m(D) \geq 2m_0$.\\

\noindent Consider the underlying graph $G'$ of $D'$.
Note that the graph $G'$ is $\kappa$-connected since  $X_{G}(v)=X_{D'}(v)$ and $\rho(G')=\rho(D')$. Since $m(D') = 2m(G')$, we have that $m(G') \geq m_0$ and by Theorem \ref{theo:remoteness-of-kappa-connected-graphs}, we have that $\rho(G') \leq \rho(PK_{n,m_0, \kappa})$. \\

\noindent Now we consider the digraph $\overleftrightarrow{PK_{n,m_0, \kappa}}$. Note that $\overleftrightarrow{PK_{n,m_0, \kappa}}$ is also $\kappa$-connected Eulerian digraph with size at least $2m_0$ since  $m(\overleftrightarrow{PK_{n,m_0, \kappa}})=2m(PK_{n,m_0, \kappa})$ and further $\rho(PK_{n,m_0, \kappa})=\rho(\overleftrightarrow{PK_{n,m_0, \kappa}})$. Hence we obtain that 

\[ \rho(D) = \rho(D') = \rho(G') 
       \leq \rho(PK_{n,m_0,\kappa})
    = \rho(\overleftrightarrow{PK_{n,m_0,\kappa}}).\]

If $m_0 \equiv \binom{n-1}{2} \pmod{\kappa}$ and
$\frac{1}{2} [ n(3\kappa-1) -2\kappa^2-\kappa+1 - b(\kappa-b) ] 
\leq m_0
\leq \binom{n-1}{2}$, the Eulerian digraph $\overleftrightarrow{PK_{n,m_0,\kappa}}$, which  has size $2m_0$, shows that the bound is sharp.
 \end{proof}
 
The following corollary is an immediate consequence of Theorem \ref{th:rho(vertex connectivity - D)}.

\begin{corollary}\label{coro:rho(D)}
Let $ D $ be an Eulerian digraph of order $ n $ and size at least $ 2m_{0} $, where $m_0 \in \mathbb{N}$. Then 
\[\rho(D) \leq \; \rho(\overleftrightarrow{PK_{n,m_0}}) \; = \dfrac{n+2}{2}-\dfrac{m_0}{n-1}. \]
This bound is sharp.
\end{corollary}
 
Next we use Theorem \ref{th:m(D)} to obtain a sharp upper bound on the remoteness of Eulerian digraphs of given order, size, and $\lambda$-edge-connectivity for $\lambda \in \{2,3\}$. We first introduce the additional notation needed to present the corresponding results for graphs proved by Dankelmann et al. \cite{DanMafMal2025}.

\begin{definition}\label{deflambda}
{\rm (Dankelmann, Mafunda, Mallu~\cite{DanMafMal2025})}\\
Let $\lambda\in\{2,3\}$. A graph $G$ is said to be a $\lambda$-edge-connected path-complete graph if there exist
$k \in \mathbb{N}\cup\{0\}$ and  $a, b \in \mathbb{N}$ with
\[ G = \left\{ \begin{array}{cc}        
  \left[K_{1}+K_{\lambda}\right]^{k}+K_{a}+K_{b} 
  & \textrm{if $k \geq 1$ and $ab \geq \lambda$, or}\\
  \left[K_{1}+K_{\lambda}\right]^{k}+K_{1}+K_{a}+K_{b} 
& \textrm{if $a \geq \lambda$, or}\\ 
\left[K_1 + K_3\right]^k + K_2 + K_a + K_1 
& \textrm{if $\lambda=3$, $k \geq 1$ and $a \geq 3$.}          
 \end{array} \right. \]
\end{definition}

For given $n,m \in \mathbb{N}$ for which there exists a $\lambda$-edge-connected path-complete graph of order $n$ and size at least $m$, Dankelmann et al. \cite{DanMafMal2025} defined $PK_{n,m}^{\lambda}$ to be such a graph of minimum size, where  $\lambda\in\{2,3\}$.

\begin{theorem} \label{theo: remoteness of edge connectivity}
{\rm (Dankelmann, Mafunda, Mallu~\cite{DanMafMal2025})}\\
Let $\lambda \in \{2,3\}$ and let $G$ be a $\lambda$-edge-connected graph of order $n$ and size at least $m$. Then 
\[ \rho(G) \leq \rho(PK^{\lambda}_{n,m}). \]
\end{theorem}

\begin{theorem}\label{theo:remoteness-of-edge-connected-graphs}
{\rm (Dankelmann, Mafunda, Mallu~\cite{DanMafMal2025})}\\
(a) Let $G$ be a $2$-edge-connected graph of order $n$ and size $m$. 
Then
\[ \rho(G) \leq \left\{ \begin{array}{cc}
   \frac{n}{3} & \textrm{if $m <\lceil \frac{5}{3}n \rceil -2$,} \\
   \frac{n}{3} - \frac{2m}{3(n-1)} +\frac{5}{3} 
         & \textrm{if $m  \geq \lceil \frac{5}{3}n \rceil -2$},
      \end{array} \right. \]
and this bound is sharp apart from an additive constant. \\
(b) Let $G$ be a $3$-edge-connected graph of order $n$ and size $m$. 
Then
\[ \rho(G) \leq \left\{ \begin{array}{cc}
   \frac{n}{4} & \textrm{if $m < \lceil \frac{9}{4}n\rceil -2$,} \\
   \frac{n}{4} - \frac{m}{2(n-1)} + \frac{3}{2} 
        & \textrm{if $m \geq \lceil \frac{9}{4}n\rceil -2$,}
      \end{array} \right. \]
and this bound is sharp apart from an additive constant. 
\end{theorem}
 
We omit the proof of the following theorem, since it is almost identical to that of Theorem \ref{th:rho(vertex connectivity - D)}.

\begin{theorem}\label{th:rho(2-edge D)}
Let $ D $ be a $\lambda$-edge-connected Eulerian digraph of order $ n $ and size at least $ 2m_{0} $, where $m_0 \in \mathbb{N}$  and $\lambda\in\{2,3\}$. Then 
\[\rho(D) \leq \; \rho(\overleftrightarrow{PK^{\lambda}_{n,m_0}}). \]
This bound is sharp.
\end{theorem}

The following corollary is an immediate consequence of Theorem \ref{th:rho(2-edge D)}.

\begin{corollary}
(a) Let $ D $ be a $2$-edge-connected Eulerian digraph of order $ n $ and size at least $ 2m_{0} $, where $m_0 \in \mathbb{N}$. Then
\[ \rho(D) \leq \left\{ \begin{array}{cc}
   \frac{n}{3} & \textrm{if $m_0 <\lceil \frac{5}{6}n \rceil -1$,} \\
   \frac{n}{3} - \frac{2m_0}{3(n-1)} +\frac{5}{3} 
         & \textrm{if $m_0  \geq \lceil \frac{5}{6}n \rceil -1$},
      \end{array} \right. \]
and this bound is sharp apart from an additive constant. \\
(b) Let $ D $ be a $3$-edge-connected Eulerian digraph of order $ n $ and size at least $ 2m_{0} $, where $m_0 \in \mathbb{N}$. Then
\[ \rho(G) \leq \left\{ \begin{array}{cc}
   \frac{n}{4} & \textrm{if $m_0 < \lceil \frac{9}{8}n\rceil -1$,} \\
   \frac{n}{4} - \frac{m_0}{2(n-1)} + \frac{3}{2} 
        & \textrm{if $m_0 \geq \lceil \frac{9}{8}n\rceil -2$,}
      \end{array} \right. \]
and this bound is sharp apart from an additive constant. 
\end{corollary}

We conclude this section by establishing an upper bound on the remoteness of Eulerian digraphs of given order and minimum degree. We further show that this bound is close to being sharp, up to an additive constant. We first present the result of Dankelmann \cite{Dan2019} that provides an upper bound on the average distance of an Eulerian digraph in terms of its order and minimum degree. We then derive from this result a corresponding upper bound on the remoteness for Eulerian digraphs of given order and minimum degree.

\begin{theorem}\label{theo: average distance of Eulerian digraph}
{\rm (Dankelmann~\cite{Dan2019})}\\
Let $D$ be an Eulerian digraph of order $n$ and minimum degree $\delta$. Then
\[
\mu(D) \le \frac{3}{2}\,\frac{n-1}{\delta+1} + \frac{3}{2} + \frac{1}{n-1}.
\]
\end{theorem}

The following corollary follows from Claim 2 in the proof of Theorem \ref{theo: average distance of Eulerian digraph}, which shows that the given bound on $\mu(D)$ is actually a bound on the average distance of any vertex \cite{Dan2019}.

\begin{corollary}\label{coro: remotness of Eulerian digraph}
Let $D$ be an Eulerian digraph of order $n$ and minimum degree $\delta$. Then
\[
\rho(D) \le \frac{3}{2}\,\frac{n-1}{\delta+1} + \frac{3}{2} + \frac{1}{n-1}.
\]
\end{corollary}

The bound in Corollary \ref{coro: remotness of Eulerian digraph} s not far from being best possible, as the following Eulerian digraph shows.

\begin{example}
Let $\delta,\ell \in \mathbb{N}$ with $\delta \ge 2$ and $\ell$ odd. Define the graph \( D_{\delta, \ell} \) as the sequential sum  

\[
D_{\delta, \ell} = \overleftrightarrow{K_{\delta}} + [\overleftrightarrow{K_1} + \overleftrightarrow{K_1} + \overleftrightarrow{K_{\delta-1}}]^{\ell} + \overleftrightarrow{K_1} + \overleftrightarrow{K_1} +\overleftrightarrow{ K_{\delta}},
\]

where the pattern \(\overleftrightarrow{K_1} + \overleftrightarrow{K_1} + \overleftrightarrow{K_{\delta-1}}\) is repeated \( \ell \) times. It is easy to check that $D_{\delta, \ell}$ has $(\ell+2)(\delta+1)$ vertices, is Eulerian, has minimum degree $\delta$ and 

\[
\rho(D_{\delta, \ell}) \geq \frac{3n}{2(\delta+1)} - \frac{3}{2}.
\]
\end{example}

\section*{Acknowledgements}
The author wishes to express sincere gratitude to Professor Peter Dankelmann and Professor Sonwabile Mafunda for their insightful suggestions on the manuscript.

\end{document}